\documentclass[10pt]{amsart}

\usepackage[english]{babel}
\usepackage{amsmath,amsfonts,amssymb,amsthm}

\usepackage[dvipsnames]{xcolor}

\usepackage{tikz}
\usetikzlibrary{shapes}
\usetikzlibrary{fit}
\usetikzlibrary{decorations.pathmorphing}
\tikzstyle{inversion}=[color=red,line width=.4mm]
\tikzstyle{coinversion}=[color=Bleu,line width=.4mm]
\tikzstyle{alors}=[line width=.4mm,->]
\tikzstyle{Classe}=[ellipse,draw=Vert!100,fill=Vert!10,thick]
\tikzstyle{ClasseB}=[ellipse,draw=Bleu!100,fill=Bleu!10,thick]
\tikzstyle{ClasseR}=[ellipse,draw=Rouge!100,fill=Rouge!10,thick]
\definecolor{pourpre}{rgb}{0.7,0,0.9}
\definecolor{vertclair}{rgb}{0.5, 1.0, 0.5}
\definecolor{vertex}{rgb}{0, 0.5, 0}
\definecolor{vertolive}{rgb}{0, 0.5, 0.0}
\definecolor{jaune}{rgb}{1.0, 1.0, 0}
\definecolor{turquoise}{rgb}{0, 1.0, 1.0}
\definecolor{orange}{rgb}{1.0, 0.6, 0}
\definecolor{bleu}{rgb}{0.25, 0.25, 0.75}
\definecolor{bleuclair}{rgb}{0.6, 0.5, 1.0}
\definecolor{rosepale}{rgb}{1.0, 0.7, 1.0}
\definecolor{brun}{cmyk}{0, 0.8, 1, 0.6}
\definecolor{noir}{RGB}{0,0,0}
\definecolor{rouge}{RGB}{205,35,38}
\definecolor{violet}{RGB}{181,18,225}
\definecolor{bleut}{RGB}{121,176,197}

\newcommand{\Cacher}[1]{}

\usepackage{cite}
\usepackage{array}
\usepackage{multicol}
\usepackage{xspace}
\usepackage{subfigure}
\title[Left  trapezoids]{Inversions and the Gog-Magog problem}
\author{Philippe Biane \qquad Hayat Cheballah}
\email{biane@univ-mlv.fr, cheballa@univ-mlv.fr}
\date{\today}
\address{Institut Gaspard-Monge, universit\'e Paris-Est Marne-la-Vall\'ee,
5 Boulevard Descartes, Champs-sur-Marne, 77454, Marne-la-Vall\'ee cedex 2,
France}
\thanks{A shorter version of this paper will appear in the proceedings of the 2013 FPSAC Conference}
\newtheorem{theorem}{Theorem}[section]
\newtheorem{prop}[theorem]{Proposition}
\newtheorem{lemma}[theorem]{Lemma}
\newtheorem{definition}[theorem]{Definition}

\newtheorem{Remarque}[theorem]{Remark}

\newtheorem{conj}[theorem]{Conjecture}

\numberwithin{equation}{section}

\renewcommand{\leq}{\leqslant}
\renewcommand{\geq}{\geqslant}

\definecolor{Noir}{RGB}{0,0,0}
\definecolor{Blanc}{RGB}{255,255,255}
\definecolor{Rouge}{RGB}{205,35,38}
\definecolor{Bleu}{RGB}{2,60,195}
\definecolor{Vert}{RGB}{23,163,1}
\definecolor{Violet}{RGB}{181,18,225}
\definecolor{Orange}{RGB}{255,113,15}
\definecolor{vertex}{rgb}{0, 0.5, 0}

\begin{document}

\begin{abstract}
We consider the  problem of finding a bijection between the sets of alternating sign matrices and of totally symmetric self complementary plane partitions, which can be reformulated using Gog and  Magog triangles. In a previous work we introduced GOGAm triangles, which are images of Magog triangles by the Sch\"utzenberger involution. In this paper
we introduce  left Gog and GOGAm trapezoids. We conjecture that they are equienumerated,  and 
 we give an explicit
 bijection between such trapezoids with one or two diagonals. We also study the distribution of inversions and coinversions in Gog triangles.
\end{abstract}

\maketitle

\section{Introduction} 


It is  well a known open problem in  combinatorics  to find a bijection between the sets of alternating sign matrices on one side, and totally symmetric 
self complementary plane partitions on the other side. One can reformulate the problem using so-called Gog and Magog triangles, which are particular species of Gelfand-Tsetlin triangles. In particular, Gog triangles are in simple bijection with alternating sign matrices of the same
size, while Magog triangles are in bijection with totally symmetric 
self complementary plane partitions. In \cite{mrr},  Mills, Robbins and Rumsey introduced trapezoids in this problem by cutting out $k$ diagonals on the right (with the conventions used in the present paper) of a triangle of size $n$, and they conjectured that Gog and Magog trapezoids of the same size are equienumerated. Some related conjectures can be found in \cite{kratt}.
Zeilberger
\cite{zeilberger} proved the conjecture of Mills, Robbins and Rumsey, but no explicit bijection is known, except for $k=1$ (which is a relatively easy problem) and for $k=2$, this bijection being the main result of \cite{BC}. In this last paper a new class of triangles and (right) trapezoids was introduced, called GOGAm triangles (or trapezoids), which are in bijection with the Magog triangles by the Sch\"utzenberger involution acting on Gelfand-Tsetlin triangles. 

In this paper we introduce a new class of trapezoids by cutting diagonals of Gog and GOGAm triangles on the left instead of the right. We conjecture that the left Gog and GOGAm trapezoids of the same shape are equienumerated, and give a bijective proof of this for  trapezoids composed of one or two diagonals.  It turns out that the bijection we obtain for left trapezoids is much simpler than the one of \cite{BC} for right trapezoids.
We also consider other shapes, such as rectangles or pentagons, for which we also conjecture that  Gog and GOGAm are equienumerated.

Inversions of Gog triangles play a fundamental role in our contruction. For this reason, we study the distribution of inversions and also of coinversions (which are inversions of the vertically symmetrized triangle) among Gog triangles. In particular we determine the pairs $(p,q)$ for which there exists Gog triangles of a given size with $p$ inversions and $q$ coinversions.

The paper is organized as follows.  

In section 2 we give some elementary definitions about  Totally Symmetric Self-Complementary Plane Partitions, Alternating Sign Matrices, Gelfand-Tsetlin triangles. In section 3, 4 and 5 we consider Gog, Magog and GOGAm triangles and trapezoids, and pentagons in section 6.
Section 7 is devoted to the formulation of  conjectures on the existence of some bijections. We consider some statistics in section 8. In section 9 we give a bijection between Gog and GOGAm trapezoids with two diagonals, and in section 10 a bijection between pentagons of shape $(n,3,3,3)$. Finally, the last section is devoted to the study of the distribution of the number of inversions and coinversions in Gog triangles.

\section{Basic definitions}

\subsection{ Totally Symmetric Self-Complementary Plane Partitions and Alternating Sign Matrices }
A \emph{Plane Partition} is a stack of cubes in a corner,
$$\includegraphics[width=3cm]{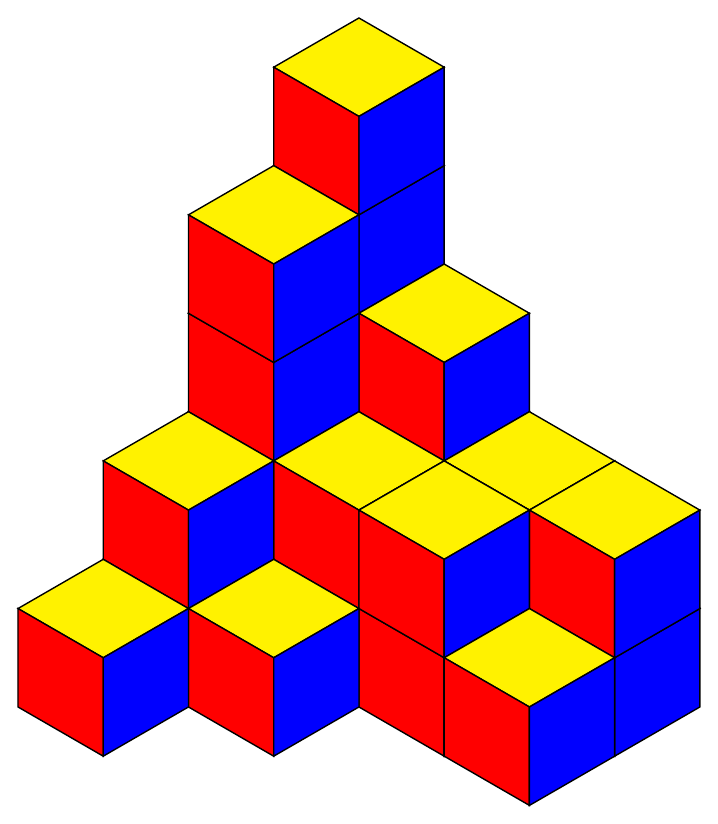}$$
Choosing a large cube that contains a plane partition, one can also encode it as a lozenge tiling of an hexagon

A \emph{Totally Symmetric Self-Complementary Plane Partition} (TSSCPP) of size $n$ is a plane partition, inside a cube of side $2n$, such that the lozenge tiling has all the symmetries of the hexagon, as in the  picture below, where $n=3$.
$$
\begin{tikzpicture}[scale=.35] 

\tikzstyle{ajuda}=[gray,very thin]
\tikzstyle{Ajuda}=[gray!30,very thin]
\tikzstyle{nilp}=[blue,thick]
\tikzstyle{nilP}=[red,thick,dotted]
\tikzstyle{Nilp}=[magenta,inner sep=0pt,minimum size=2.5pt,fill,shape=circle]
\tikzstyle{NilP}=[red,inner sep=0pt,minimum size=2pt,fill,shape=circle]
\tikzstyle{nclp}=[fill,shape=circle,inner sep=0pt,minimum size=.3mm]
\tikzstyle{ASM}=[black]
\tikzstyle{ASMleve}=[black!20]
\tikzstyle{FPL}=[very thick,blue,rounded corners=3pt]
\tikzstyle{NILP}=[very thick,red,rounded corners=1pt]
\tikzstyle{fpl}=[thick,blue,rounded corners=2pt]
\tikzstyle{tl}=[thick,red,rounded corners=6pt]
\tikzstyle{RS}=[thick,black,rounded corners=3pt]
\tikzstyle{TL}=[thick,blue,rounded corners=6pt]
\tikzstyle{boucle}=[very thick,blue]
\tikzstyle{boucle out}=[very thick,red,dashed]
\tikzstyle{up}=[thin,blue,->]
\tikzstyle{down}=[thin,blue,<-]
\tikzstyle{boucl}=[thick,blue]
\tikzstyle{Boucle}=[red]
\tikzstyle{FPLeve}=[very thick,blue!20,rounded corners=3pt]
\tikzstyle{seta}=[thick,red]
\tikzstyle{zface}=[fill=yellow,draw=black,very thin]
\tikzstyle{xface}=[fill=pourpre!50,draw=white,very thin]
\tikzstyle{yface}=[fill=vertclair!40,draw=white,very thin]
 \tikzstyle{zFace}=[fill=blue!80,draw=black,very thin]
 \tikzstyle{xFace}=[fill=red!80,draw=black,very thin]
 \tikzstyle{yFace}=[fill=yellow!80,draw=black,very thin]
 
 \def \quadz{--+(1,0,0)--+(1,1,0)--+(0,1,0)--cycle};
 \def \quady{--+(1,0,0)--+(1,0,1)--+(0,0,1)--cycle};
 \def \quadx{--+(0,1,0)--+(0,1,1)--+(0,0,1)--cycle};
 \foreach \x/\y/\z in 
{0/0/6,1/0/6,2/0/6,3/0/5,4/0/5,5/0/3,
0/1/6,1/1/5,2/1/5,3/1/4,4/1/3,5/1/1,
0/2/6,1/2/5,2/2/4,3/2/3,4/2/2,5/2/1,
0/3/5,1/3/4,2/3/3,3/3/2/,4/3/1,5/3/0,
0/4/5,1/4/3,2/4/2,3/4/1,4/4/1,5/4/0,
0/5/3,1/5/1,2/5/1,3/5/0,4/5/0,5/5/0
}{
  \filldraw[yFace,x={(-.866,-.5)},z={(-1,-1)}] (\y,\z,\x) \quady;
  \filldraw[xFace,x={(-.866,-.5)},z={(-1,-1)}] (\z,\y,\x) \quadx;
  \filldraw[zFace,x={(-.866,-.5)},z={(-1,-1)}] (\x,\y,\z) \quadz;}

\end{tikzpicture}
$$

An \emph{Alternating Sign Matrix} (ASM) is a square matrix with entries in
$\{-1,0,+1\}$ such that, on each line and on each column, the non zero entries alternate in sign, the sum of each line and
each column being equal to 1.

It is well known \cite{andrews} that the number of TSSCPP of size $n$ is 

\begin{align}
\label{eq.AnEnum}
A_n&=\prod_{j=0}^{n-1} \frac{(3j+1)!}{(n+j)!} =1,2,7,42,429,\ldots
\end{align}

and, 
as first proved by Zeilberger \cite{zeilberger}, this is also the number of ASM of size $n$
(more about this story is  in~\cite{bressoud}). 
However no explicit bijection between these classes of objects has ever been constructed, and finding one is a major open problem in 
combinatorics.
In this paper we use 
Gog and Magog triangles (defined below) in order to investigate the problem.

\subsection{Gelfand-Tsetlin triangles}
\begin{definition}
A  Gelfand-Tsetlin  triangle of size $n$ is a triangular array 
$X=(X_{i,j})_{n\geqslant i\geqslant j\geqslant 1}$ of positive integers
$$\begin{matrix}
 X_{n,1}& &X_{n,2}& &\ldots& &X_{n,n-1}& &X_{n,n}\\
  &X_{n-1,1}& &X_{n-1,2}& &\ldots& &X_{n-1,n-1}&  \\
  & &\ldots& &\ldots& &\ldots& & \\
  & & &X_{2,1}& &X_{2,2}& & & \\
  & & & &X_{1,1}& & & & \\
  \end{matrix}
$$
 such that
$$X_{i+1,j}\leqslant  X_{i,j}\leqslant X_{i+1,j+1}\quad \text{for $n-1\geqslant i\geqslant j\geqslant 1$}.$$
\end{definition}

The set of all Gelfand-Tsetlin triangles of size $n$ is a poset for the order such that $X\leq Y$ iff $X_{ij}\leq Y_{ij}$ for all $i,j$.
It is also a lattice for this order, the infimum and supremum being taken entrywise: $\max(X,Y)_{ij}=\max(X_{ij},Y_{ij})$.
\subsection{Sch\"utzenberger involution}

Gelfand-Tsetlin triangles label bases of irreducible representations of general linear groups. As such, they are in simple bijection with 
semi-standard Young tableaux. It follows that the Sch\"utzenberger involution, which is defined on SSYTs, can be transported to 
Gelfand-Tsetlin triangles. The following description of this involution has been studied by Berenstein and Kirillov \cite{kirillov}.

First define involutions $s_k$, for $k\leqslant n-1$, acting on the set of Gelfand-Tsetlin triangles of size $n$.
If $X=(x_{i,j})_{n\geqslant i\geqslant j\geqslant 1}$ is such a triangle the action of $s_k$ on $X$ is given by 
$s_kX=(\tilde X_{i,j})_{n\geqslant i\geqslant j\geqslant 1}$ with
$$\begin{array}{rcl}
\tilde X_{i,j}&=&X_{i,j},\qquad \text{if\ }i\ne k\\
\tilde X_{k,j}&=&\max(X_{k+1,j},X_{k-1,j-1})+\min(X_{k+1,j+1},X_{k-1,j})-X_{i,j}
\end{array}$$
It is understood that 
$\max(a,b)=\max(b,a)=a$ and $\min(a,b)=\min(b,a)=a$ if the entry $b$ of the triangle is not defined.
The geometric meaning of the transformation of an entry is the following: on row $k$, any entry
$X_{k,j}$ is surrounded by four (or less if it is on the boundary) numbers, increasing from left to right.
$$\begin{matrix}
 X_{k+1,j}& &X_{k+1,j+1}\\
  &X_{k,j}&\\
  X_{k-1,j-1}& &X_{k-1,j} \\
 \end{matrix}
$$
 These four numbers determine a smallest interval containing
$X_{k,j}$, namely $$\left[\max(X_{k+1,j},X_{k-1,j-1}),\min(X_{k+1,j+1},X_{k-1,j})\right]$$ and the transformation maps $X_{k,j}$ to its mirror image 
with respect to the center of this interval.

Define $\omega_j=s_js_{j-1}\ldots s_2s_1$.

\begin{definition}

The Sch\"utzenberger involution, acting on Gelfand-Tsetlin triangles of size $n$, is  given by the formula
$$S=\omega_1\omega_2\ldots\omega_{n-1}$$
\end{definition}
It is a non trivial result that $S$ is an involution \cite{kirillov}, 
and coincides with the Sch\"utzenberger involution when transported to SSYTs; 
beware that the $s_k$ {\sl do not} satisfy the braid relations.\\
\section{Gog triangles and trapezoids}
\label{sec.gog}

\begin{definition} A Gog triangle of size $n$ is a Gelfand-Tsetlin  triangle 
such that

\medskip

$(i)\qquad \qquad \qquad \qquad \qquad  X_{i,j}<X_{i,j+1}, \qquad j<i\leqslant n-1$

\medskip

in other words, such that its rows are strictly increasing, and such that

\medskip

$(ii)\qquad \qquad \qquad \qquad \qquad
X_{n,j}=j,\qquad 1\leqslant j\leqslant n$.

\medskip

\end{definition}

It is immediate to check that the set of Gog triangles of size $n$ is a sublattice of the Gelfand-Tsetlin triangles.

There is  a simple bijection between Gog triangles and Alternating sign matrices (see e.g. \cite{bressoud}). If $(M_{ij})_{1\leqslant i,j\leqslant n}$
is an ASM of size $n$,
then the matrix $\tilde M_{ij}= \sum_{k=i}^nM_{ij}$ has exactly $i-1$ entries 0 and $n-i+1$ entries 1 on row $i$. Let
$(X_{ij})_{j=1,\ldots,i}$ be the columns (in increasing order) with a 1 entry of $\tilde M$ on row $n-i+1$. The triangle 
$X=(X_{ij})_{n\geqslant i\geqslant j\geqslant 1}$ is the Gog triangle corresponding to $M$. 

For example, the following matrix $M$ is an alternating sign matrix of size $5$. We also show the matrix $\bar M$, and its associated  Gog triangle 
$$
M=\begin{pmatrix}
0&1&0&0&0\\
0&0&1&0&0\\
1&-1&0&0&1\\
0&1&-1&1&0\\
0&0&1&0&0
\end{pmatrix}
\qquad\bar M=
\begin{pmatrix}
1&1&1&1&1\\
1&0&1&1&1\\
1&0&0&1&1\\
0&1&0&1&0\\
0&0&1&0&0
\end{pmatrix}
$$

$$\begin{matrix}
 1& &2& &3& &4& &5\\
  &1& &3& &4& &5& \\
  & &1& &4& &5& & \\
  & & &2& &4& & & \\
  & & & &3& & & & \\
 \end{matrix}
$$

 \subsubsection{Gog Trapezoids}

\begin{definition}
A $(n,k)$  right Gog trapezoid   (for $k\leq n$) is an array of positive integers 
$X=(X_{i,j})_{n\geqslant i\geqslant j\geqslant 1; i-j\leq k-1}$ formed from the $k$ rightmost SW-NE diagonals of some Gog triangle of size $n$.
\end{definition}
Below is a $(5,2)$ right Gog trapezoid.

\begin{center}
\begin{tikzpicture}[scale=.5]
\tikzstyle{box}=[-,line width=0.3mm,color=red]
\tikzstyle{ntrap}=[color=red]
\draw(6,4) node{$4$} ;\draw(8,4) node{$5$} ;
\draw(5,3) node{$4$} ;\draw(7,3) node{$5$} ;
\draw(4,2) node{$3$} ;\draw(6,2) node{$4$} ;
\draw(3,1) node{$1$} ;\draw(5,1) node{$3$} ;
\draw(4,0) node{$2$} ;

\end{tikzpicture}
\end{center}

\begin{definition}
A $(n,k)$ left Gog  trapezoid   (for $k\leq n$) is an array of positive integers
$X=(X_{i,j})_{n\geqslant i\geqslant j\geqslant 1; k\geqslant j}$  formed from the $k$ leftmost NW-SE diagonals of a Gog triangle of size $n$.

\end{definition}
A more direct way of checking that a left Gelfand-Tsetlin 
trapezoid is a left Gog trapezoid is to verify that its rows are strictly increasing and that its SW-NE diagonals are bounded by $1,2,\ldots,n$.

Below is a $(5,2)$ left Gog trapezoid.
\begin{center}\label{trapeze::Gog::gauche::52}
\begin{tikzpicture}[scale=.5]
\tikzstyle{box}=[-,line width=0.3mm,color=red]
\tikzstyle{ntrap}=[color=red]
\draw  (0,4) node{$1$} ;\draw  (2,4) node{$2$} ;
\draw  (1,3) node{$1$} ;\draw(3,3) node{$3$} ;
\draw  (2,2) node{$2$} ;\draw(4,2) node{$3$} ;
\draw(3,1) node{$2$} ;\draw(5,1) node{$4$} ;
\draw(4,0) node{$4$} ;

\end{tikzpicture}
\end{center}
There is a simple involution $X\to\tilde X$ on Gog triangles of size $n$, given by

\begin{equation}\label{reflection}
\tilde X_{i,j}=n+1-X_{i,i+1-j}
\end{equation}
 which exchanges left and right trapezoids of the same size.
This involution corresponds to a vertical symmetry of the associated ASMs.
\subsubsection{Minimal completion}
Since the set of Gog triangles is a lattice, given a (left or right) Gog trapezoid, there exists a smallest   Gog triangle from which it can be extracted. We call this Gog triangle the {\sl canonical completion} of the (left or right) Gog trapezoid. Their explicit value is computed in the next Proposition.
\begin{prop}$ $
\begin{enumerate}
 \item Let $X$ be a $(n,k)$ right Gog trapezoid, then its canonical completion satisfies
$$X_{ij}=j,\quad \text{for $i\geqslant j+k$}$$
\item Let $X$ be a $(n,k)$ left Gog trapezoid, then its canonical completion satisfies
$$X_{i,j}=\max(X_{i,k}+j-k,X_{i-1,k}+j-k-1,\ldots, X_{i-j+k,k})\quad\text{for $j\geqslant k$}$$
\end{enumerate}
\end{prop}

\begin{proof} The first case  (right trapezoids) is trivial, the formula for the second case (left trapezoids) 
is easily proved  by induction on $j-k$.
\end{proof}

\medskip

For example, the $(5,2)$ left Gog trapezoid above has canonical completion
\begin{center}
\begin{tikzpicture}[scale=.5]
\tikzstyle{box}=[-,line width=0.3mm,color=red]
\tikzstyle{ntrap}=[color=red]
\draw  (0,4) node{$1$} ;\draw  (2,4) node{$2$} ;
\draw[ntrap](4,4) node{$3$} ;
\draw[ntrap](6,4) node{$4$} ;\draw[ntrap](8,4) node{$5$} ;
\draw  (1,3) node{$1$} ;\draw(3,3) node{$3$} ;
\draw[ntrap](5,3) node{$4$} ;\draw[ntrap](7,3) node{$5$} ;
\draw  (2,2) node{$2$} ;\draw(4,2) node{$3$} ;
\draw[ntrap](6,2) node{$4$} ;
\draw(3,1) node{$2$} ;\draw(5,1) node{$4$} ;
\draw(4,0) node{$4$} ;

\end{tikzpicture}
\end{center}

Remark that the supplementary entries of the canonical completion of a left Gog trapezoid depend only on its rightmost NW-SE diagonal. 
\medskip

 The right trapezoids defined above coincide (modulo easy reindexations) with those of Mills, Robbins, Rumsey \cite{mrr}, and Zeilberger \cite{zeilberger}. They are in obvious bijection with the ones in \cite{BC} (actually the Gog trapezoids of \cite{BC} are the canonical completions of the right Gog trapezoids defined above). 
\section{Magog triangles and trapezoids} 
\begin{definition}
A Magog triangle of size $n$ is a Gelfand-Tsetlin triangle such that 
$X_{jj}\leq j$  for all $1\leq j\leq n$.
\end{definition}

\begin{definition}
A $(n,k)$  right Magog trapezoid   (for $k\leq n$) is an array of positive integers 
$X=(X_{i,j})_{n\geqslant i\geqslant j\geqslant 1; i-j\leq k-1}$ formed from the $k$ rightmost SW-NE diagonals of some Magog triangle of size $n$.
\end{definition}
Below is a $(5,2)$ right Magog trapezoid.
\begin{center}
\begin{tikzpicture}[scale=.5]
\tikzstyle{box}=[-,line width=0.3mm,color=red]
\tikzstyle{ntrap}=[color=red]
\draw(6,4) node{$2$} ;\draw(8,4) node{$3$} ;
\draw(5,3) node{$1$} ;\draw(7,3) node{$3$} ;
\draw(4,2) node{$1$} ;\draw(6,2) node{$2$} ;
\draw(3,1) node{$1$} ;\draw(5,1) node{$2$} ;
\draw(4,0) node{$1$} ;

\end{tikzpicture}
\end{center}
\subsubsection{Minimal completion}
 Given a  right Magog trapezoid, there exists a smallest   Magog triangle from which it can be extracted. It is obtained by putting $1$'s on the triangle sitting NW of the Magog trapezoid.
For example
\begin{center}
\begin{tikzpicture}[scale=.5]
\tikzstyle{box}=[-,line width=0.3mm,color=red]
\tikzstyle{ntrap}=[color=red]
\draw[ntrap]  (0,4) node{$1$} ;\draw[ntrap]   (2,4) node{$1$} ;\draw[ntrap] (4,4) node{$1$} ;
\draw[ntrap]   (2,2) node{$1$} ;
\draw[ntrap]   (1,3) node{$1$} ;\draw[ntrap] (3,3) node{$1$} ;
\draw(5,3) node{$1$} ;\draw(7,3) node{$3$} ;

\draw(6,4) node{$2$} ;\draw(8,4) node{$3$} ;
\draw(4,2) node{$1$} ;\draw(6,2) node{$2$} ;
\draw(3,1) node{$1$} ;\draw(5,1) node{$2$} ;
\draw(4,0) node{$1$} ;

\end{tikzpicture}
\end{center}

\section{GOGAm triangles and trapezoids}

\begin{definition}
A GOGAm triangle of size $n$ is a Gelfand-Tsetlin triangle 
whose image by the Sch\"utzenberger involution is a Magog triangle (of size $n$).

\end{definition}

It is shown in \cite{BC} that GOGAm triangles are the Gelfand-Tsetlin triangles 
 $X=(X_{i,j})_{n\geqslant i\geqslant j\geqslant 1}$ 
such that 
$X_{nn}\leq n$ and, for all $1\leq k\leq n-1$,
and all $n=j_0> j_1> j_2\ldots>j_{n-k}\geq 1$, one has
\begin{equation}\label{GOGAm}
\left(\sum_{i=0}^{n-k-1}X_{j_i+i, j_i}-X_{j_{i+1}+i,j_{i+1}}\right)+X_{j_{n-k}+n-k,j_{n-k}}\leq k
\end{equation}

The problem of finding an explicit bijection between Gog and Magog triangles can therefore be reduced to that of finding an explicit bijection between Gog and GOGAm triangles. 

\begin{definition}
A $(n,k)$ right GOGAm    trapezoid   (for $k\leq n$) is an array of positive integers $X=(x_{i,j})_{n\geqslant i\geqslant j\geqslant 1; i-j\leq k-1}$
formed from the $k$ rightmost SW-NE diagonals of a GOGAm triangle of size $n$.
\end{definition}
Below is a $(5,2)$ right GOGAm trapezoid

\begin{center}
\begin{tikzpicture}[scale=.5]
\tikzstyle{box}=[-,line width=0.3mm,color=red]
\tikzstyle{ntrap}=[color=red]
\draw(6,4) node{$2$} ;\draw(8,4) node{$4$} ;
\draw(5,3) node{$2$} ;\draw(7,3) node{$4$} ;
\draw(4,2) node{$2$} ;\draw(6,2) node{$4$} ;
\draw(3,1) node{$1$} ;\draw(5,1) node{$4$} ;
\draw(4,0) node{$3$} ;

\end{tikzpicture}
\end{center}
\begin{definition}
A $(n,k)$ left GOGAm trapezoid (for $k\leq n$) is  an array of positive integers $X=(x_{i,j})_{n\geqslant i\geqslant j\geqslant 1; k\geqslant j}$
formed from the $k$ leftmost NW-SE diagonals of a GOGAm trapezoid of size $n$.
\end{definition}

A $(5,2)$ left GOGAm trapezoid

\begin{center}
\begin{tikzpicture}[scale=.5]
\tikzstyle{box}=[-,line width=0.3mm,color=red]
\tikzstyle{ntrap}=[color=red]
\draw  (0,4) node{$1$} ;\draw  (2,4) node{$1$} ;
\draw  (1,3) node{$1$} ;\draw(3,3) node{$2$} ;
\draw  (2,2) node{$1$} ;
\draw(4,2) node{$2$} ;
\draw(3,1) node{$2$} ;\draw(5,1) node{$3$} ;
\draw(4,0) node{$3$} ;

\end{tikzpicture}
\end{center}
\subsubsection{Minimal completion}
The set of GOGAm triangles is not a sublattice of the Gelfand-Tsetlin triangles, nevertheless, given a (right or left) GOGAm trapezoid, we shall see that there exists a smallest GOGAm triangle which extends it. We call canonical completion this triangle.
\begin{prop}\label{comp}$ $
\begin{enumerate}
 \item Let $X$ be a $(n,k)$ right GOGAm trapezoid, then its canonical completion is given by
$$X_{ij}=1\quad \text{for $n\geqslant i\geqslant j+k$}$$
\item Let $X$ be a $(n,k)$ left GOGAm trapezoid, then its canonical completion is given by
$$X_{i,j}=X_{i-j+k,k}\quad \text{for $n\geqslant i\geqslant j\geqslant k$}$$
in other words, the  added entries are  constant on SW-NE diagonals
\end{enumerate}
\end{prop}

\begin{proof}
In both cases, the completion above is the smallest Gelfand-Tsetlin triangle containing the trapezoid, therefore 
it is enough to check that if $X$ is a $(n,k)$ right or left GOGAm trapezoid, then its completion, 
as indicated in the proposition \ref{comp},  is a GOGAm triangle. The claim follows from the following lemma.

\begin{lemma} Let $X$ be a GOGAm triangle.
\begin{itemize}
 \item[{\bf i)}]The triangle obtained from $X$ by replacing the entries on the upper left triangle $(X_{ij}, n\geqslant i\geqslant j+k)$ by 1 is a GOGAm triangle. 
 \item[{\bf ii)}] Let $n\geqslant m\geqslant k\geqslant1$.
If $X$ is constant on each partial SW-NE diagonal $(X_{i+l,k+l}; n-i\geqslant l\geqslant 0)$ for $i\geqslant m+1$  then the triangle  obtained from $X$ by replacing the entries  $(X_{m+l,k+l};  n-m\geqslant l\geqslant 1)$ by $X_{m,k}$ is a GOGAm triangle.
\end{itemize}
\end{lemma}
\begin{proof}
It is easily seen that the above replacements give a Gelfand-Tsetlin triangle.
Both proofs then follow  by inspection of the formula (\ref{GOGAm}), 
which shows   that, upon making the above replacements, the quantity on the left cannot increase. 
 \end{proof}

{\sl End of proof of Proposition \ref{comp}.}
The case of right GOGAm trapezoids is dealt with by part {\bf i)} of the preceding Lemma.
 The case of left trapezoids follows by replacing successively the SW-NE partial diagonals  as in part {\bf ii)} of the Lemma. \end{proof}

\medskip

Canonical completion of a  $(5,2)$ left GOGAm trapezoid:

\begin{center}
\begin{tikzpicture}[scale=.5]
\tikzstyle{box}=[-,line width=0.3mm,color=red]
\tikzstyle{ntrap}=[color=red]
\draw  (0,4) node{$1$} ;\draw  (2,4) node{$1$} ;
\draw[ntrap](4,4) node{$2$} ;
\draw[ntrap](6,4) node{$2$} ;\draw[ntrap](8,4) node{$3$} ;
\draw  (1,3) node{$1$} ;\draw(3,3) node{$2$} ;
\draw[ntrap](5,3) node{$2$} ;\draw[ntrap](7,3) node{$3$} ;
\draw  (2,2) node{$1$} ;
\draw(4,2) node{$2$} ;
\draw[ntrap](6,2) node{$3$} ;
\draw(3,1) node{$2$} ;\draw(5,1) node{$3$} ;
\draw(4,0) node{$3$} ;

\end{tikzpicture}
\end{center}

Let $X=(X_{i,j})_{n\geqslant i\geqslant j\geqslant 1}$ be a Gelfand-Tsetlin triangle, and $k$ such that
$X_{i,j}=1$ for $i-j\geq k$. Let $Y$ be the image of $X$ by the 
 Sch\"utzenberger involution, it follows easily from the description of the operations $s_i$ that 
$Y_{i,j}=1$ for $i-j\geq k$.
in particular, the image, by the 
 Sch\"utzenberger involution, of the canonical completion of a right Magog $(n,k)$ trapezoid is the   canonical completion of a right GOGAm $(n,k)$ trapezoid, and vice versa. It follows that 
the Sch\"utzenberger involution induces a bijection between right Magog and GOGAm trapezoids.
\section{Pentagons}
\begin{definition}
For $(n,k,l,m)$, with $n\geq k,l,m$,  a $(n,k,l,m)$ Gog (resp. GOGAm) pentagon is  an array of positive integers $X=(x_{i,j})_{n\geqslant i\geqslant j\geqslant 1; k\geqslant j;j+l\geqslant i+1}$
formed from the intersection of the $k$ leftmost NW-SE diagonals, the $l$  rightmost SW-NE  diagonals and the $m$ bottom lines 
 of a Gog (resp. GOGAm) triangle  of size $n$.
\end{definition}

Remark that if $m\geq k+l+1$ then the pentagon is  a rectangle, whereas if $m\leq k,l$ then it is a Gelfand-Tsetlin triangle of size $m$.
Similarly to the case of trapezoids, 
one can check that a $(n,k,l,m)$ Gog (or GOGAm) pentagon
 has a canonical (i.e. minimal) completion as  a $(n,k)$ left  trapezoid, or as a 
$(n,l)$ right trapezoid, and finally as a  triangle of size $n$.

\section{Results and conjectures}
\begin{theorem}[Zeilberger \cite{zeilberger}]
For all $k\leq n$, the $(n,k)$ right Gog and Magog  trapezoids are equienumerated
\end{theorem}

Composing by the Sch\"utzenberger involution yields that for all $k\leq n$, the $(n,k)$ right Gog and GOGAm  trapezoids are equienumerated.
In \cite{BC} a bijective proof of this last fact is given for $(n,1)$ and $(n,2)$ right trapezoids.

\begin{conj}\label{conjecture::equipotence::gog::GOGAm::nk::gauche}
For all $k\leq n$, the $(n,k)$ left Gog and GOGAm   trapezoids are equienumerated.
\end{conj}
In the section \ref{bij} we will give a bijective proof of this conjecture for $(n,1)$ and $(n,2)$ trapezoids. Remark that the right and left Gog trapezoids of   shape $(n,k)$ are equienumerated (in fact a simple bijection between them was given above as (\ref{reflection})).
\begin{conj}
For any $n,k,l,m$  the
$(n,k,l,m)$ Gog and GOGAm pentagons are equienumerated.
\end{conj}
We will give a bijective proof for $(n,3,3,3)$ pentagons (which, in this case, are actually triangles).

If we consider left GOGAm  trapezoids as GOGAm triangles, using the canonical completion, then we can take their image by the Sch\"utzenberger involution and obtain a subset of the Magog triangles, for each $(n,k)$. It seems however that this subset does not have a simple direct characterization. This shows that GOGAm triangles and trapezoids are a useful tool in  the bijection problem between Gog and Magog triangles.
\section{statistics}
We will define three statistics on the Gog, Magog and GOGAm triangles.
\subsection{The $\alpha $ statistics}
For a Gog triangle $X$, we define $\alpha_{Gog}(X)$ as the number of indices $k$ such that  $X_{k,1}=1$.
For a GOGAm triangle $X$, let $l\geq 0$ be the largest integer such that $X_{n-l+1,1}=1$. If there exists an inversion of the triangle $X$ of the form $(n-l+k,k+1)$ with $1\leq k\leq l-1$ then  we define 
$\alpha_{GOGAm}(X)=l-1$ If there is no such inversion, we put $\alpha_{GOGAm}(X)=l$.
For a Magog triangle $X$ we put
$\alpha_{Magog}(X)=\alpha_{GOGAm}(S(X))$ where $S$ is the Sch\"utzenberger involution.
\subsection{The $\beta$ statistics}
For a Gog triangle $X$ we define $\beta_{Gog}(X)=X_{1,1}$. For a GOGAm triangle we put
$\beta_{GOGAm}(X)=X_{1,1}$. For a Magog triangle 
$\beta_{Magog}(X)=\sum_{i=1}^n X_{n,i}-\sum_{i=1}^{n-1}X_{n-1,i}$.
It is a well known property of the Sch\"utzenberger involution that
$\beta_{GOGAm}(S(X))=
\beta_{Magog}(X)$ for any Magog triangle $X$.
\subsection{The $\gamma $ statistics}
For a Gog triangle $X$, of size $n$, we define $\gamma _{Gog}(X)$ as the number of indices $k$ such that  $X_{k,k}=n$.

Let $X$ be a Magog triangle. Let $k$ be the largest integer such that $X_{k,k}=k$. We define  a sequence of pairs $(i_l,j_l)_{1\leq l\leq k}$ by $(i_1,j_1)=(k,k)$ and, if $k>1$,

$(i_{l+1},j_{l+1})=(i_l,j_l-1)$ if $X_{i_l,j_l-1}=X_{i_l-1,j_l-1}$, 

$(i_{l+1},j_{l+1})=(i_l-1,j_l-1)$ if $X_{i_l,j_l-1}<X_{i_l-1,j_l-1}$

Since $j$ decreases by 1 at each step, 
the sequence ends at step $k$, when $j_{k}=1$.
We put $\gamma _{Magog}(X)=i_k$.

If $X$ is a GOGAm triangle we put $\gamma _{GOGAm}(X)=\gamma _{Magog}(S(X))$.

\begin{conj}
For any $n$ the three statistics $\alpha,\beta,\gamma $ are equienumerated on, respectively, Gog, Magog, and GOGAm triangles.
\end{conj}

We remark that the statistics $\alpha,\beta,\gamma $ are equienumerated on Magog and GOGAm triangles, by construction.

Finally, if we identify  Gog triangles with alternating sign matrices, then the three statistics $\alpha_{Gog},\beta_{Gog},\gamma_{Gog} $ correspond,  to the position of the 1 in respectively, the leftmost column, the bottom line, and the rightmost column. Some recent results on the joint enumeration of these statistics can be found in \cite{ayerromik}.
\section{Bijections between Gog and GOGAm left trapezoids}\label{bij}
\subsection{$(n,1)$ left trapezoids}
The sets of  $(n,1)$  left Gog trapezoids and of  $(n,1)$  left GOGAm trapezoids  both coincide with the set of nondecreasing sequences $X_{n,1},\ldots, X_{1,1}$ satisfying
$$X_{j,1}\leq n-j+1$$ (note that these sets  are counted by Catalan numbers). Therefore the identity map provides a trivial bijection between these two sets.

\subsection{$(n,2)$ left trapezoids}
In order to treat the $(n,2)$ left  trapezoids we will recall some definitions from \cite{BC}.
\subsubsection{Inversions}  

\begin{definition}\label{inversions}

An  inversion in a Gelfand-Tsetlin triangle $X$ is a pair $(i,j)$ such that $X_{i,j}=X_{i+1,j}$.

\end{definition}
For example the following Gog triangle contains three inversions, $(2,2)$, $(3,1)$, $(4,1)$, the respective equalities being depicted on the  picture:

\bigskip

\begin{center}

 \begin{tikzpicture}[scale=0.5]
      \draw  (1,3) node{$1$} ;\draw  (3,3) node{$2$} ;\draw  (5,3) node{$3$} ;\draw  (7,3) node{$4$} ;\draw  (9,3) node{$5$} ;
            \draw  (2,2) node{$1$} ;\draw  (4,2) node{$3$} ;\draw  (6,2) node{$4$} ;\draw  (8,2) node{$5$} ;
                  \draw  (3,1) node{$1$} ;\draw  (5,1) node{$4$} ;\draw  (7,1) node{$5$} ;
                        \draw  (4,0) node{$2$} ;\draw  (6,0) node{$4$} ;
                                 \draw  (5,-1) node{$3$} ;

\draw[line width=0.3mm,color=red] (1.9,2.2)--(1.1,2.85);\draw[line width=0.3mm,color=red] (2.9,1.2)--(2.1,1.85);
\draw[line width=0.3mm,color=red] (5.9,.2)--(5.1,.85);
\end{tikzpicture}

\end{center}

The name {\sl inversion} comes from the fact that, for a Gog triangle corresponding to an alternating sign matrix which is a permutation, its inversions are in one to one correspondance with the inversions of the permutation.

\begin{definition}
Let  $X=(X_{i,j})_{n\geqslant i\geqslant j\geqslant 1}$ be a Gog triangle and let $(i,j)$ be such that
 $1\leqslant j\leqslant i\leqslant n$. 

An inversion $(k,l)$ covers $(i,j)$
 if $i=k+p$ and $j=l+p$ for some $p$ with $1\leqslant p\leqslant n-k$. 

\end{definition}
 The entries $(i,j)$ covered by an inversion are depicted with $"+"$ on the following picture.
\begin{center}

 \begin{tikzpicture}[scale=0.5]
      \draw  (1,3) node{$\circ$} ;\draw  (3,3) node{$\circ$} ;\draw  (5,3) node{$\circ$} ;\draw  (7,3) node{+} ;\draw  (9,3) node{$\circ$} ;
            \draw  (2,2) node{$\circ$} ;\draw  (4,2) node{$\circ$} ;\draw  (6,2) node{+} ;\draw  (8,2) node{$\circ$} ;
                  \draw  (3,1) node{$\circ$} ;\draw  (5,1) node{$\circ$} ;\draw  (7,1) node{$\circ$} ;
                        \draw  (4,0) node{$\circ$} ;\draw  (6,0) node{$\circ$} ;
                                 \draw  (5,-1) node{$\circ$} ;

\draw[line width=0.3mm,color=red] (4.15,1.8)--(4.9,1.2);

\end{tikzpicture}

\end{center}
\subsubsection{Standard procedure}
The basic idea for our bijection 
 is that for any inversion in the Gog triangle we should subtract 1 from the entries covered by this inversion, scanning the inversions along the successive NW-SE diagonals, starting from the rightmost diagonal, and scanning each diagonals from NW to SE. We call this the {\sl standard procedure}. If the successive triangles obtained after each of these steps are Gelfand-Tsetlin triangles, then we say that the initial triangle is {\sl admissible}.

\begin{prop}
Let $X$ be a Gog triangle of size $n$, then the triangle obtained by applying the standard procedure to $X$ is a GOGAm triangle of size $n$.
\end{prop}
\begin{proof}
Let us denote by $Y$ the triangle obtained from $X$ by the standard procedure. One has 
$Y_{i,j}=X_{i,j}-c_{i,j}$ where $c_{i,j}\geq 0$ is the number of inversions which are covered
 by $(i,j)$. Note that this number weakly increases along SW-NE diagonals:
$c_{i+1,j+1}\geq c_{i,j}$.

 We have to prove that for  all $n=j_0> j_1> j_2\ldots>j_{n-k}\geq 1$, one has
$$
\left(\sum_{i=0}^{n-k-1}Y_{j_i+i, j_i}-Y_{j_{i+1}+i,j_{i+1}}\right)+Y_{j_{n-k}+n-k,j_{n-k}}\leq k
$$

Let us rewrite the sum on the left hand side as
\begin{eqnarray*}
S&=&X_{n,n}-c_{n,n}-
\sum_{i=1}^{n-k}(X_{j_i+i-1, j_i}-c_{j_i+i-1,j_i}-X_{j_i+i,j_i}+c_{j_i+i,j_i})\\
&=&X_{n,n}-\left[\sum_{i=1}^{n-k}(X_{j_i+i-1, j_i}-X_{j_i+i,j_i})
+(c_{j_{i-1}+i-1,j_{i-1}}-c_{j_i+i-1,j_i})\right]-c_{j_{n-k}+n-k,j_{n-k}}.
\end{eqnarray*}
One has $X_{n,n}=n$, furthermore, for each $i$ one has either $$X_{j_i+i-1, j_i}-X_{j_i+i,j_i}\geq 1$$ or $$X_{j_i+i-1, j_i}=X_{j_i+i,j_i}$$
 in which case $(j_i+i-1, j_i)$ is an inversion, therefore
$$c_{j_{i-1}+i-1,j_{i-1}}-c_{j_i+i-1,j_i}\geq 1.$$ It follows that for each term in the sum 
$$(X_{j_i+i-1, j_i}-X_{j_i+i,j_i})
+(c_{j_{i-1}+i-1,j_{i-1}}-c_{j_i+i-1,j_i})\geq 1$$  therefore $$S\leq n-(n-k)-c_{j_{n-k}+n-k,j_{n-k}}\leq k.$$
\end{proof}

For admissible triangles the standard procedure is invertible, meaning that the  triangle can recovered uniquely from its associated GOGAm triangle. Indeed it suffices for that to scan the inversions of the GOGAm triangle  in the reverse order and to add one to the entries covered by each inversion, in order to recover the original admissible triangle.

The following property is not difficult to check, we leave it as an exercise to the reader.
\begin{prop}
The Gog triangles corresponding to  permutation matrices, in the correspondance between alternating sign matrices and Gog triangles, are all admissible.
\end{prop}

Like in \cite{BC} the bijection between left Gog and GOGAm trapezoids will be obtained by a modification of the Standard Procedure.

\subsubsection{Characterization of $(n,2)$ GOGAm trapezoids}
The family of  inequalities (\ref{GOGAm}) simplifies in the case of $(n,2)$ GOGAm trapezoids, indeed if we identify such a trapezoid with its canonical completion, then most of the terms in the left hand side are zero, so that these inequalities reduce to

\begin{eqnarray}
X_{i,2}&\leqslant& n-i+2\label{GOGAm21}\\
X_{i,2}-X_{i-1,1}+X_{i,1}&\leqslant &n-i+1\label{GOGAm22}
\end{eqnarray}

Remark that,
 since $-X_{i-1,1}+X_{i,1}\leqslant 0$, the inequality (\ref{GOGAm22}) follows from (\ref{GOGAm21}) unless $X_{i-1,1}=X_{i,1}$.
\subsubsection{From Gog to GOGAm}
Let $X$ be a $(n,2)$ left Gog trapezoid. We shall construct a  $(n,2)$ left GOGAm trapezoid $Y$ by scanning the inversions in the leftmost NW-SE diagonal of $X$, starting from NW. Let us denote by $n> i_1>\ldots>i_k\geq 1$ these inversions, so that 
$X_{i,1}=X_{i+1,1}$ if and only if $i\in\{i_1,\ldots,i_k\}$. We also put $i_0=n$. We will  construct  a sequence of $(n,2)$ left Gelfand-Tsetlin
 trapezoids $X=Y^{(0)},Y^{(1)},Y^{(2)},\ldots,Y^{(k)}=Y$. 

Let us  assume that we have constructed the trapezoids up to $Y^{(l)}$,   that 
$Y^{(l)}\leq X$,  that $Y^{(l)}_{ij}= X_{ij} $  for $i\leqslant i_{l} $, and that 
inequalities  (\ref{GOGAm21}) and (\ref{GOGAm22}) are satisfied by $Y^{(l)}$ for $i\geqslant i_l+1$.
This is the case for $l=0$.

Let 
$m$ be the largest integer such that $Y^{(l)}_{m,2}=Y^{(l)}_{i_{l+1}+1,2}$.
We put
$$\begin{array}{rclcl}
 Y^{(l+1)}_{i,1}&=&Y^{(l)}_{i,1}&\text{ for }& n\geqslant i\geqslant m\ \text{and}\ i_{l+1}\geqslant i\\
Y^{(l+1)}_{i,1}&=&Y^{(l)}_{i+1, 1}&\text{ for }& m-1\geqslant i> i_{l+1}\\  
Y^{(l+1)}_{i,2}&=&Y^{(l)}_{i,2}&\text{ for }& n\geqslant i\geqslant m+1\ \text{and}\ i_{l+1}\geqslant i\\
Y^{(l+1)}_{i,2}&=&Y^{(l)}_{i,2}-1&\text{ for }& m\geqslant i\geqslant i_{l+1}+1.
\end{array}$$
From the definition of $m$, and the fact that $X$ is a Gog trapezoid, we see that this new triangle is a Gelfand-Tsetlin triangle,   that 
$Y^{(l+1)}\leq X$, and that $Y^{(l+1)}_{ij}= X_{ij} $  for $i\leqslant i_{l+1}$.
Let us now check that the  trapezoid $Y^{(l+1)}$ satisfies the inequalities  (\ref{GOGAm21}) and (\ref{GOGAm22})
for $i\geqslant i_{l+1}+1$.
The first series of inequalities, for $i\geqslant i_{l+1}+1$, follow from the fact that $Y^{(l)}\leq X$.
For the second series, they are satisfied for $i\geqslant m+1$ since this is the case for $Y^{(l)}$. For $m\geqslant i\geqslant i_{l+1}+1$, observe that $$Y^{(l+1)}_{i,2}-Y^{(l+1)}_{i+1,1}+Y^{(l+1)}_{i+1,1}\leqslant Y^{(l+1)}_{i,2}= Y^{(l+1)}_{m,2}=Y^{(k)}_{m,2}-1\leqslant n-m+1$$  by (\ref{GOGAm21}) for $Y(l)$, from which 
(\ref{GOGAm22}) follows.

This proves that $Y^{(l+1)} $ again satisfies the induction hypothesis. Finally $Y=Y^{(k)}$ is a GOGAm triangle: indeed inequalities (\ref{GOGAm21}) follow again from $Y^{(l+1)}\leqslant X$, and (\ref{GOGAm22}) for $i\leq i_{k}$ follow from the fact that there are no inversions in this range. It follows that
the above algorithm provides a map from $(n,2)$ left Gog trapezoids to 
$(n,2)$ left  GOGAm trapezoids.
Observe that the number of inversions in the leftmost diagonal of $Y$ is the same as for $X$, but the positions of these inversions are not the same in general.
\subsubsection{Inverse map}\label{invert}
We now describe the inverse map, from GOGAm left trapezoids to Gog left trapezoids.

We start from an $(n,2)$ GOGAm left trapezoid $Y$, and construct a sequence 
$$Y=Y^{(k)},Y^{(k-1)},Y^{(k-2)},\ldots, Y^{(0)}=X$$ of intermediate Gelfand-Tsetlin trapezoids.

Let $n-1\geqslant \iota_1>\iota_2\ldots>\iota_k\geqslant 1$ be the inversions of the leftmost diagonal of $Y$, and let $\iota_{k+1}=0$.
 Assume that $Y^{(l)}$ has been constructed and that $Y_{ij}^{(l)}=Y_{ij}$ for $i-j\geqslant \iota_{l+1}$. 
This is the case for $l=k$.

Let $p$ be the smallest integer such that
$Y^{(l)}_{i_{l}+1,2}=Y^{(l)}_{p,2}$. We put
$$\begin{array}{rclcl}
 Y^{(l-1)}_{i,1}&=&Y^{(l)}_{i,1}&\text{ for }& n\geqslant i\geqslant \iota_{l}+1\ \text{and}\ p\geqslant i\\
Y^{(l-1)}_{i,1}&=&Y^{(l)}_{i-1,1}&\text{ for }& \iota_l\geqslant i\geqslant p\\  
Y^{(l-1)}_{i,2}&=&Y^{(l)}_{i,2}&\text{ for }& n\geqslant i\geqslant \iota_l+2\ \text{and}\ p-1\geqslant i\\
Y^{(l-1)}_{i,2}&=&X^{(l)}_{i,2}-1&\text{ for }& \iota_l+1\geqslant i\geqslant p.
\end{array}$$

It is immediate to check that if $X$ is an $(n,2)$ left Gog trapezoid, and $Y$ is its image by the first algorithm then 
the above algorithm applied to $Y$ yields $X$ back, actually the sequence $Y^{(l)}$ is the same. Therefore in order to prove the bijection we only need to show that if 
$Y$ is a $(n,2)$ left GOGAm trapezoid then the algorithm is well defined and $X$ is a Gog left trapezoid. This is a bit cumbersome, but not difficult, and very similar to the opposite case, so we leave this task to the reader.
\subsubsection{The $\alpha$ statistic}\label{stats}
Observe that in our bijection the value of the bottom entry $X_{1,1}$ is unchanged when we go from Gog to GOGAm trapezoids.
The same was true of the bijection in \cite{BC} for right trapezoids. 
Actually we make the following conjecture, which extends Conjecture \ref{conjecture::equipotence::gog::GOGAm::nk::gauche} above.

\begin{conj}
For each $n,k,l$ the
$(n,k)$ left Gog and GOGAm trapezoids with bottom entry $X_{1,1}=l$ are equienumerated.

\end{conj}
\subsection{An example}
In this section we work out an example of the algorithm from the Gog trapezoid $X$  to the  GOGAm trapezoid $Y$ by showing the successive trapezoids $Y^{(k)}$.
At each step we indicate the  inversion, as well as the entry covered by this inversion, and the values of the parameters $i_l, p$.
The algorithm also runs backwards to yield the GOGAm$\to$Gog bijection.
\medskip

\begin{tikzpicture}[scale=0.6]
 \tikzstyle{barr}= [opacity=.2,line width=6 mm,cap=round,color=red]
\tikzstyle{barb}= [opacity=.2,line width=6 mm,cap=round,color=blue]
\tikzstyle{barv}= [opacity=.2,line width=6 mm,cap=round,color=violet]
\tikzstyle{barvr}= [opacity=.2,line width=6 mm,cap=round,color=vertex]
\tikzstyle{comp}= [color=vertex]
\tikzstyle{ntrap}=[color=red]
\tikzstyle{inv}=[color=vertex,line width=0.3mm]
\draw[ClasseR](3,10)ellipse(6mm and 6mm);
\node (11)at (1,10){$1$};\node (21)at (2,9) {$1$} ;\node  (31) at (3,8){$1$} ;
\node(41)at  (4,7) {$2$} ;\node (51)at  (5,6) {$3$} ;
\node (61)at  (6,5) {$3$} ;
\node (71)at  (7,4) {$3$} ;
\node (12)at  (3,10) {$2$} ;
\node (22)at  (4,9) {$2$} ;
\node (32)at (5,8){$4$};
\node (42)at (6,7){$4$};
\node(52)at  (7,6) {$4$} ;
\node(62)at  (8,5) {$4$} ;

 \draw[inv](11)--(21);
\node at (12,10){$X=Y^{(0)}$};
\node at (12,8){$i_1=6$};
\node at (12,7){$m=7$};
\end{tikzpicture}

\begin{tikzpicture}[scale=0.6]
 \tikzstyle{barr}= [opacity=.2,line width=6 mm,cap=round,color=red]
\tikzstyle{barb}= [opacity=.2,line width=6 mm,cap=round,color=blue]
\tikzstyle{barv}= [opacity=.2,line width=6 mm,cap=round,color=violet]
\tikzstyle{barvr}= [opacity=.2,line width=6 mm,cap=round,color=vertex]
\tikzstyle{comp}= [color=vertex]
\tikzstyle{inv}=[color=vertex,line width=0.3mm]
\tikzstyle{ntrap}=[color=red]
\draw[ClasseR](4,9)ellipse(6mm and 6mm);
\node (11)at (1,10){$1$};\node (21)at (2,9) {$1$} ;\node  (31) at (3,8){$1$} ;
\node(41)at  (4,7) {$2$} ;\node (51)at  (5,6) {$3$} ;
\node (61)at  (6,5) {$3$} ;
\node (71)at  (7,4) {$3$} ;
\node (12)at  (3,10) {$1$} ;
\node (22)at  (4,9) {$2$} ;
\node (32)at (5,8){$4$};
\node (42)at (6,7){$4$};
\node(52)at  (7,6) {$4$} ;
\node(62)at  (8,5) {$4$} ;

 \draw[inv](21)--(31);
\node at (12,10){$Y^{(1)}$};
\node at (12,8){$i_2=5$};
\node at (12,7){$m=6$};
\end{tikzpicture}

\begin{tikzpicture}[scale=0.6]
 \tikzstyle{barr}= [opacity=.2,line widt Fischer, Ilse; Romik, Dan More refined enumerations of alternating sign matrices. Adv. Math. 222 (2009), no. 6, 2004–2035.h=6 mm,cap=round,color=red]
\tikzstyle{barb}= [opacity=.2,line width=6 mm,cap=round,color=blue]
\tikzstyle{barv}= [opacity=.2,line width=6 mm,cap=round,color=violet]
\tikzstyle{barvr}= [opacity=.2,line width=6 mm,cap=round,color=vertex]
\tikzstyle{comp}= [color=vertex]
\tikzstyle{ntrap}=[color=red]
\tikzstyle{inv}=[color=vertex,line width=0.3mm]
\draw[ClasseR](7,6)ellipse(6mm and 6mm);
\node (11)at (1,10){$1$};\node (21)at (2,9) {$1$} ;\node  (31) at (3,8){$1$} ;
\node(41)at  (4,7) {$2$} ;\node(51)at  (5,6) {$3$} ;
\node (61)at  (6,5) {$3$} ;
\node (71)at  (7,4) {$3$} ;
\node (12)at  (3,10) {$1$} ;
\node (22)at  (4,9) {$1$} ;
\node (32)at (5,8){$4$};
\node (42)at (6,7){$4$};
\node(52)at  (7,6) {$4$} ;
\node(62)at  (8,5) {$4$} ;

 \draw[inv](51)--(61);
\node at (12,10){$Y^{(2)}$};
\node at (12,8){$i_3=2$};
\node at (12,7){$m=5$};
\end{tikzpicture}

\begin{tikzpicture}[scale=0.6]
 \tikzstyle{barr}= [opacity=.2,line width=6 mm,cap=round,color=red]
\tikzstyle{barb}= [opacity=.2,line width=6 mm,cap=round,color=blue]
\tikzstyle{barv}= [opacity=.2,line width=6 mm,cap=round,color=violet]
\tikzstyle{barvr}= [opacity=.2,line width=6 mm,cap=round,color=vertex]
\tikzstyle{comp}= [color=vertex]
\tikzstyle{ntrap}=[color=red]
\tikzstyle{inv}=[color=vertex,line width=0.3mm]
\draw[ClasseR](8,5)ellipse(6mm and 6mm);
\node (11)at (1,10){$1$};\node (21)at (2,9) {$1$} ;\node  (31) at (3,8){$1$} ;
\node(41)at  (4,7) {$1$} ;\node (51)at  (5,6) {$2$} ;
\node (61)at  (6,5) {$3$} ;
\node (71)at  (7,4) {$3$} ;
\node (12)at  (3,10) {$1$} ;
\node (22)at  (4,9) {$1$} ;
\node (32)at (5,8){$3$};
\node (42)at (6,7){$3$};
\node(52)at  (7,6) {$3$} ;
\node(62)at  (8,5) {$4$} ;

 \draw[inv](61)--(71);

\node at (12,10){$Y^{(3)}$};
\node at (12,8){$i_4=1$};
\node at (12,7){$m=2$};
\end{tikzpicture}

\begin{tikzpicture}[scale=0.6]
 \tikzstyle{barr}= [opacity=.2,line width=6 mm,cap=round,color=red]
\tikzstyle{barb}= [opacity=.2,line width=6 mm,cap=round,color=blue]
\tikzstyle{barv}= [opacity=.2,line width=6 mm,cap=round,color=violet]
\tikzstyle{barvr}= [opacity=.2,line width=6 mm,cap=round,color=vertex]
\tikzstyle{comp}= [color=vertex]
\tikzstyle{ntrap}=[color=red]
\tikzstyle{inv}=[color=vertex,line width=0.3mm]
\node (11)at (1,10){$1$};\node (21)at (2,9) {$1$} ;\node  (31) at (3,8){$1$} ;
\node(41)at  (4,7) {$1$} ;\node (51)at  (5,6) {$2$} ;
\node (61)at  (6,5) {$3$} ;
\node (71)at  (7,4) {$3$} ;
\node (12)at  (3,10) {$1$} ;
\node (22)at  (4,9) {$1$} ;
\node (32)at (5,8){$3$};
\node (42)at (6,7){$3$};
\node(52)at  (7,6) {$3$} ;
\node(62)at  (8,5) {$3$} ;

\node at (12,7){$Y^{(4)}=Y$};
\end{tikzpicture}
\section{The $(n,3,3,3)$ pentagons}
As we remarked above, the $(n,3,3,3)$ Gog or GOGAm pentagons are actually Gelfand-Tsetlin triangles composed of positive integers, of the form:
\begin{figure}[ht]
\begin{center}
    \begin{tikzpicture}[scale=.45]
        \begin{scope}
                   \draw  (3,1) node{a} ;\draw  (5,1) node{d} ;\draw  (7,1) node{f} ;
                         \draw  (4,0) node{b} ;\draw  (6,0) node{e} ;
                         \draw  (5,-1) node{c} ;
  \end{scope}
           \end{tikzpicture}
\end{center}
\end{figure}
\smallskip

In order that such a Gelfand-Tsetlin triangle be a $(n,3,3,3)$ Gog pentagon, it is necessary and sufficent that:

$$a<d<f;\quad b<e,\quad
f\leq n.$$

\medskip

For GOGAm pentagons the conditions are:
$$f\leq n;\quad f-e+d\leq n-1;\quad f-c+b\leq n-1;\quad f-e+d-b+a\leq n-2.$$

We will now describe the bijection from Gog to  GOGAm by decomposing the set of Gog pentagons according to the inversion pattern. In the left column we put the different Gog pentagons, on the right the corresponding GOGAm pentagons. There are three possible places for inversions, hence $2^3=8$ cases to consider. In 7 of these cases, the triangle is admissible, and we apply the standard procedure. there is only one case where this procedure does not apply.

Verifying that this table gives a bijection between Gog and GOGAm pentagons of this form is straightforward, but tedious, so we leave this task to the interested reader.

\bigskip

\begin{center}
\begin{tabular}{|c|c|}
\hline
Gog&GOGAm
\\
\hline
\begin{tikzpicture}[scale=.4]
        
                   \draw  (3,1) node{a} ;\draw  (5,1) node{d} ;\draw  (7,1) node{f} ;
                         \draw  (4,0) node{b} ;\draw  (6,0) node{e} ;
                         \draw  (5,-1) node{c} ;

 \end{tikzpicture}
&
 \begin{tikzpicture}[scale=.4]
                   \draw  (3,1) node{a} ;\draw  (5,1) node{d} ;\draw  (7,1) node{f} ;
                         \draw  (4,0) node{b} ;\draw  (6,0) node{e} ;
                         \draw  (5,-1) node{c} ;
 \end{tikzpicture}
\\
\hline

\begin{tikzpicture}[scale=.4]
       
                   \draw  (3,1) node{a} ;\draw  (5,1) node{d} ;\draw  (7,1) node{f} ;
                         \draw  (4,0) node{b} ;\draw  (6,0) node{d} ;
                         \draw  (5,-1) node{c} ;
\draw[inversion] (5.15,.85)--(5.85,.15);
\end{tikzpicture}
&
   \begin{tikzpicture}[scale=.4]
                   \draw  (3,1) node{a} ;\draw  (5,1) node{d} ;\draw  (7,1) node{f-1} ;
                         \draw  (4,0) node{b} ;\draw  (6,0) node{d} ;
                         \draw  (5,-1) node{c} ;

           \end{tikzpicture}
\\
\hline

\begin{tikzpicture}[scale=.4]
       
                   \draw  (3,1) node{a} ;\draw  (5,1) node{d} ;\draw  (7,1) node{f} ;
                         \draw  (4,0) node{a} ;\draw  (6,0) node{e} ;
                         \draw  (5,-1) node{c} ;
\draw[inversion] (3.15,.85)--(3.85,.15);
   \end{tikzpicture}
        &
\begin{tikzpicture}[scale=.4]
                   \draw  (3,1) node{a} ;\draw  (5,1) node{d-1} ;\draw  (7,1) node{f} ;
                         \draw  (4,0) node{a} ;\draw  (6,0) node{e} ;
                         \draw  (5,-1) node{c} ;
   \end{tikzpicture}
\\
\hline
\begin{tikzpicture}[scale=.4]
       
                   \draw  (3,1) node{a} ;\draw  (5,1) node{d} ;\draw  (7,1) node{f} ;
                         \draw  (4,0) node{b} ;\draw  (6,0) node{e} ;
                         \draw  (5,-1) node{b} ;
\draw[inversion] (4.15,-0.15)--(4.85,-0.85);
\end{tikzpicture}
  &
       \begin{tikzpicture}[scale=.4]
                   \draw  (3,1) node{a} ;\draw  (5,1) node{d} ;\draw  (7,1) node{f-1} ;
                         \draw  (4,0) node{b} ;\draw  (6,0) node{e-1} ;
                         \draw  (5,-1) node{b} ;
  
           \end{tikzpicture}

\\
\hline
\end{tabular}
\hskip 1cm
\begin{tabular}{|c|c|}
\hline
Gog&GOGAm
\\
\hline

\begin{tikzpicture}[scale=.4]
       
                   \draw  (3,1) node{a} ;\draw  (5,1) node{d} ;\draw  (7,1) node{f} ;
                         \draw  (4,0) node{b} ;\draw  (6,0) node{d} ;
                         \draw  (5,-1) node{b} ;
\draw[inversion] (4.15,-0.15)--(4.85,-0.85);
\draw[inversion] (5.15,.85)--(5.85,.15);
 \end{tikzpicture}
&
        \begin{tikzpicture}[scale=.4]
                   \draw  (3,1) node{a} ;\draw  (5,1) node{d-1} ;\draw  (7,1) node{f-2} ;
                         \draw  (4,0) node{a} ;\draw  (6,0) node{d-1} ;
                         \draw  (5,-1) node{b} ;
  
           \end{tikzpicture}
\\
\hline
\begin{tikzpicture}[scale=.4]
      
                   \draw  (3,1) node{a} ;\draw  (5,1) node{d} ;\draw  (7,1) node{f} ;
                         \draw  (4,0) node{a} ;\draw  (6,0) node{d} ;
                         \draw  (5,-1) node{c} ;
\draw[inversion] (3.15,.85)--(3.85,.15);
\draw[inversion] (5.15,.85)--(5.85,.15);
   \end{tikzpicture}
&
        \begin{tikzpicture}[scale=.4]
                   \draw  (3,1) node{a} ;\draw  (5,1) node{d-1} ;\draw  (7,1) node{f-1} ;
                         \draw  (4,0) node{a} ;\draw  (6,0) node{d} ;
                         \draw  (5,-1) node{c} ;
  
           \end{tikzpicture}
\\
\hline
 \begin{tikzpicture}[scale=.4]
     
                   \draw  (3,1) node{a} ;\draw  (5,1) node{d} ;\draw  (7,1) node{f} ;
                         \draw  (4,0) node{a} ;\draw  (6,0) node{e} ;
                         \draw  (5,-1) node{a} ;
\draw[inversion] (3.15,.85)--(3.85,.15);
\draw[inversion] (4.15,-0.15)--(4.85,-0.85);
\end{tikzpicture}
&
       \begin{tikzpicture}[scale=.4]
                   \draw  (3,1) node{a} ;\draw  (5,1) node{d-1} ;\draw  (7,1) node{f-1} ;
                         \draw  (4,0) node{a} ;\draw  (6,0) node{e-1} ;
                         \draw  (5,-1) node{a} ;
  
           \end{tikzpicture}
\\
\hline
\begin{tikzpicture}[scale=.4]
       
                   \draw  (3,1) node{a} ;\draw  (5,1) node{d} ;\draw  (7,1) node{f} ;
                         \draw  (4,0) node{a} ;\draw  (6,0) node{d} ;
                         \draw  (5,-1) node{a} ;
\draw[inversion] (4.15,-0.15)--(4.85,-0.85);
\draw[inversion] (5.15,.85)--(5.85,.15);
\draw[inversion] (3.15,.85)--(3.85,.15);
 \end{tikzpicture}
        &
\begin{tikzpicture}[scale=.4]
                   \draw  (3,1) node{a} ;\draw  (5,1) node{d-1} ;\draw  (7,1) node{f-2} ;
                         \draw  (4,0) node{a} ;\draw  (6,0) node{d-1} ;
                         \draw  (5,-1) node{a} ;
  
           \end{tikzpicture}
\\
\hline
\end{tabular}

\end{center}

We have also obtained a bijection between $(n,3,3,5)$ Gog and GOGAm pentagons, along the same lines. However it is too long to be reported here. Note that these pentagons are actually squares as below

\begin{figure}[ht]
\begin{center}
    \begin{tikzpicture}[scale=.45]
        \begin{scope}\draw  (5,3) node{g};
\draw  (4,2) node{d} ;\draw  (6,2) node{h} ;
                   \draw  (3,1) node{a} ;\draw  (5,1) node{e} ;\draw  (7,1) node{i} ;
                         \draw  (4,0) node{b} ;\draw  (6,0) node{f} ;
                         \draw  (5,-1) node{c} ;
  \end{scope}
           \end{tikzpicture}
\end{center}
\end{figure}

\section{On the distribution of inversions and coinversions}

\subsection{}
We recall the definition of inversions (cf Definition \ref{inversions}), and introduce the dual notion of coinversion.
\begin{definition}
\smallskip

An  inversion in a Gog triangle $X$ is a pair $(i,j)$ such that $X_{i,j}=X_{i+1,j}$.

\smallskip

 A coinversion is a pair $(i,j)$ such that $X_{i,j}=X_{i+1,j+1}$.
 
\end{definition}

For example, the  Gog triangle in~\eqref{eq::inversion::coinversion::triangle::monotone} 
contains three inversions, $(2,2)$, $(3,1)$, $(4,1)$ 
and five coinversions, $(3,2)$, $(3,3)$, $(4,2)$, $(4,3)$, $(4,4)$. 
\begin{equation}\label{eq::inversion::coinversion::triangle::monotone}
    \begin{split}
    \begin{tikzpicture}[scale=.5]
       \draw  (1,3) node{1} ;\draw  (3,3) node{2} ;\draw  (5,3) node{3} ;\draw  (7,3) node{4} ;\draw  (9,3) node{5} ;
             \draw  (2,2) node{1} ;\draw  (4,2) node{3} ;\draw  (6,2) node{4} ;\draw  (8,2) node{5} ;
                   \draw  (3,1) node{1} ;\draw  (5,1) node{4} ;\draw  (7,1) node{5} ;
                         \draw  (4,0) node{2} ;\draw  (6,0) node{4} ;
                         \draw  (5,-1) node{3} ;
  \draw[inversion] (1.9,2.2)--(1.1,2.85);\draw[inversion] (2.9,1.2)--(2.1,1.85);
  \draw[inversion] (5.9,.2)--(5.1,.85);
  \draw[coinversion](8.1,2.15)--(8.9,2.85);\draw[coinversion](7.1,1.15)--(7.9,1.85);
  \draw[coinversion](5.1,1.15)--(5.9,1.85);\draw[coinversion](6.1,2.15)--(6.9,2.85);
  \draw[coinversion](4.1,2.15)--(4.9,2.85);
  \end{tikzpicture}
    \end{split}
 \end{equation}

We denote by $\mu(X)$ (resp. $\nu(X)$) the number of inversions (resp. coinversions) of a Gog triangle $X$. Since a pair $(i,j)$ cannot be an inversion and a coinversion at the same time in a Gog triangle, one has
$$\nu(X)+\mu(X)\leq\frac{n(n-1)}{2}$$
Actually one can easily see that 
$\frac{n(n-1)}{2}-\nu(X)-\mu(X)$ is the number of $-1$'s in the alternating sign matrix associated to the Gog triangle $X$. Also inversions and coinversions correspond to different types of vertices in the six vertex model, see e.g. \cite{BDFZJ12}.

 Table~\ref{tab:inv::coinv::3::4} below shows the joint distribution of $\mu$ and $\nu$,
  for~$n=3\mbox{ and } n=4$.
\begin{table}[h!]
\centering
\subtable[]{
\centering
\begin{tabular}{c|cccc|}
 &0&1&2&3\\
 \hline
0& & & &1\\
1& &1&2&\\
2& &2&&\\
3&1 &&&\\
\hline
\end{tabular}
\label{tab:inv::coinv::3}
}
$\quad$
\subtable[]{
\centering
\begin{tabular}{c|ccccccc|}
&0&1&2&3&4&5&6\\
\hline
0& & & & & & &1\\
1& & & &1&2&3&\\
2& & & &6&5& &\\
3& &1&6&6& & &\\
4& &2&5& & & &\\
5& &3& & & & &\\
6&1& & & & & &\\
\hline
\end{tabular}
\label{tab:inv::coinv::4}
}
\caption{The number of Gog triangles of size~$3$ (a) and~$4$ (b) with~$k$ 
inversions (horizontal values) and~$l$ coinversions (vertical values).}
\label{tab:inv::coinv::3::4}
\end{table}

The distribution is symmetric, this follows from the symmetry (\ref{reflection}). Also
we remark that the numbers on the antidiagonal are the Mahonian numbers counting permutations according to the number of their inversions.

Let us denote by~$Z(n,x,y)$  the generating function 
 of Gog triangles of size $n$ according to $\nu$ 
and $\mu$.
\begin{equation}
 Z(n,x,y)=\sum_{X\in \mathcal X_{n}}x^{\nu(X)}y^{\mu(X)}.
\end{equation}
\smallskip

The following formula has been proved in \cite{BDFZJ12}, using properties of the six vertex model.
 \begin{prop}
  
 \begin{equation}
 Z(n,x,y) =
\det_{0\le i,j\le n-1}\left(-y^i\delta_{i,j+1} +
\sum_{k=0}^{\min(i,j+1)}{i-1\choose i-k}{j+1\choose k}x^k\right).  
 \end{equation}
\end{prop}
 For example, for Gog triangles of size $3$, we have
 \begin{equation}
 Z(3,x,y)=\det\begin{pmatrix}1&1&1\\ -y+x&2x&3x\\x&-y^2+2x+x^2&3x+3x^2\end{pmatrix} =x^3+xy+y^3+2x^2y+2xy^2.
\end{equation}
which matches part $(a)$ of Table 1.

It is however not so easy to use this formula in order to prove results on the distribution of inversion and coinversions.

\subsection{}
Let us denote by $A_{n,k}$ the set of pairs of nonnegative integers $(i,j)$ such that $$i\geq \frac{k(k+1)}{2}, \qquad
j\geq \frac{(n-k-1)(n-k)}{2},\qquad i+j\leq \frac{n(n-1)}{2}$$ and let
$$A_n=\cup_{k=0}^{n-1}A_{n,k}.$$

 We will give a simple combinatorial proof of the following.
\begin{theorem}\label{diamond}
There exists a Gog triangle with $i$ inversions and $j$ coinversions if and only if $(i,j)$ belongs to the set $A_n$. If $i=\frac{k(k+1)}{2}$ and 
$j=\frac{(n-k-1)(n-k)}{2}$ for some $k\in [0,n-1]$ then this triangle is unique, furthermore its bottom value is $n-k$.
\end{theorem}

\begin{Remarque}\label{pp+1}
Note, for future reference, that if $(l,m)$ belongs to the set $A_n$ and if
$l<\frac{p(p+1)}{2}$ then $m\geq\frac{(n-p)(n-p+1)}{2}$.
\end{Remarque}
\subsection{Proof of Theorem \ref{diamond}}
\subsubsection{Existence}
First we show that there exists a triangle of size $n$ with 
$i=\frac{k(k+1)}{2}$ inversions and 
$j=\frac{(n-k-1)(n-k)}{2}$ coinversions. Indeed the triangle is defined by
\begin{eqnarray}\label{trdiamond}
X_{ij}&=&j\quad\text{for}\quad j\leq  i-n+k\\
X_{ij}&=&n+j-i\quad\text{for}\quad j\geq k+1\\
X_{ij}&=&n-k+2j-i-1\quad\text{for}\quad i-n+k+1\leq j\leq k
\end{eqnarray}

The bottom entry of this triangle is $n-k$, as expected.

\smallskip

We give an example below: for $n=6$ and $k=3$, the triangle has $6$ inversions and $3$ coinversions:

\begin{center}
    \begin{tikzpicture}[scale=.45]

\draw[inversion](0.15,3.85)--(0.85,3.15);\draw[inversion](2.15,3.85)--(2.85,3.15);
\draw[inversion](4.15,3.85)--(4.85,3.15);
    \draw[inversion](1.15,2.85)--(1.85,2.15);\draw[inversion](2.15,1.85)--(2.85,1.15);
    \draw[inversion](3.15,2.85)--(3.85,2.15);
    \draw[coinversion](9.85,3.85)--(9.15,3.15);\draw[coinversion](8.85,2.85)--(8.15,2.15);
    \draw[coinversion](7.85,3.85)--(7.15,3.15);
 \draw  (0,4) node{1} ;\draw  (2,4) node{2} ;\draw  (4,4) node{3} ;\draw  (6,4) node{4} ;\draw  (8,4) node{5} ;\draw  (10,4) node{6} ;
    \draw  (1,3) node{1} ;\draw  (3,3) node{2} ;\draw  (5,3) node{3} ;\draw  (7,3) node{5} ;\draw  (9,3) node{6} ;
             \draw  (2,2) node{1} ;\draw  (4,2) node{2} ;\draw  (6,2) node{4} ;\draw  (8,2) node{6} ;
                   \draw  (3,1) node{1} ;\draw  (5,1) node{3} ;\draw  (7,1) node{5} ;
                         \draw  (4,0) node{2} ;\draw  (6,0) node{4} ;
                         \draw  (5,-1) node{3} ;

  \end{tikzpicture}
\end{center}

Observe that the entries which are neither inversions nor coinversions form a rectangle of size $k\times (n-k-1)$ at the bottom of the triangle.

The ASM corresponding to such a triangle has a diamond shape:

$$\begin{pmatrix}
0&0&0&1&0&0\\
0&0&1&-1&1&0\\
0&1&-1&1&-1&1\\
1&-1&1&-1&1&0\\
0&1&-1&1&0&0\\
0&0&1&0&0&0
\end{pmatrix}
$$

Starting from this triangle, it is not difficult, for a pair of integers $(l,m)$ such that
$l\geq \frac{k(k+1)}{2}, m\geq \frac{(n-k-1)(n-k)}{2}$, and $l+m\leq \frac{n(n-1)}{2}$, to construct (at least) one triangle with $l$ inversions and $m$ coinversions, for example one can add inversions by decreasing some entries, starting from the  westmost corner   of the rectangle, and add coinversions by increasing entries, starting from the eastmost corner. Here is an example with $n=6, l=9,m=5$, details of the general case are left to the reader.
\begin{figure}[ht]
\begin{center}
    \begin{tikzpicture}[scale=.45]

    \begin{scope}
\draw[inversion](0.15,3.85)--(0.85,3.15);\draw[inversion](2.15,3.85)--(2.85,3.15);
\draw[inversion](4.15,3.85)--(4.85,3.15);
    \draw[inversion](1.15,2.85)--(1.85,2.15);\draw[inversion](2.15,1.85)--(2.85,1.15);
    \draw[inversion](3.15,2.85)--(3.85,2.15);
\draw[inversion](4.15,1.85)--(4.85,1.15);
\draw[inversion](3.15,0.85)--(3.85,0.15);
\draw[inversion](4.15,-0.15)--(4.85,-0.85);
    \draw[coinversion](9.85,3.85)--(9.15,3.15);\draw[coinversion](8.85,2.85)--(8.15,2.15);
    \draw[coinversion](7.85,3.85)--(7.15,3.15);
  \draw[coinversion](6.85,2.85)--(6.15,2.15);
  \draw[coinversion](7.85,1.85)--(7.15,1.15);
 \draw  (0,4) node{1} ;\draw  (2,4) node{2} ;\draw  (4,4) node{3} ;\draw  (6,4) node{4} ;\draw  (8,4) node{5} ;\draw  (10,4) node{6} ;
    \draw  (1,3) node{1} ;\draw  (3,3) node{2} ;\draw  (5,3) node{3} ;\draw  (7,3) node{5} ;\draw  (9,3) node{6} ;
             \draw  (2,2) node{1} ;\draw  (4,2) node{2} ;\draw  (6,2) node{5} ;\draw  (8,2) node{6} ;
                   \draw  (3,1) node{1} ;\draw  (5,1) node{2} ;\draw  (7,1) node{6} ;
                         \draw  (4,0) node{1} ;\draw  (6,0) node{4} ;
                         \draw  (5,-1) node{1} ;
  \end{scope}

  \end{tikzpicture}
\end{center}

\end{figure}
\smallskip

\subsubsection{Operations on Gog triangles: projection}
\label{sec::operation}
In order to prove the only if part  of the Theorem, as well as the uniqueness statement, we now introduce two {\em standardization} operations.
These operations  build a Gog triangle of size~$n-1$ 
from a Gog triangle of size~$n$. We start by defining a {\sl projection}.

\smallskip

Given a Gog triangle~$X$, denote by
$PX$ 
the Gelfand-Tsetlin triangle of size $n-1$ obtained from $X$ 
by cutting its top row, e.g.

\begin{center}
    \begin{tikzpicture}[scale=.45]
    \begin{scope}       
      \draw  (2,2) node{1} ;\draw  (4,2) node{2} ;\draw  (6,2) node{3} ;\draw  (8,2) node{4} ;
         \draw  (3,1) node{1} ;\draw  (5,1) node{2} ;\draw  (7,1) node{3} ;
            \draw  (4,0) node{2} ;\draw  (6,0) node{3} ;
	      \draw  (5,-1) node{2} ;
    \end{scope}
    \begin{scope}[xshift=8.5cm]
    \draw[alors](0,.5)--(4,.5);     
    \end{scope}
     \begin{scope}[xshift=10cm,yshift=.5cm]
        \draw  (3,1) node{1} ;\draw  (5,1) node{2} ;\draw  (7,1) node{3} ;
            \draw  (4,0) node{2} ;\draw  (6,0) node{3} ;
	      \draw  (5,-1) node{2} ;
	      %
    \end{scope}   
         \end{tikzpicture}                    
\end{center}

\subsubsection{Left standardization}
Let $X$ be a Gog triangle of size $n$, and $PX$ its projection, which is a Gelfand-Tsetlin triangle of size $n-1$, with upper row of the form
$1,2,\ldots,k,k+2,\ldots,n$ for some $k\in[1,n]$.
For $j\leq k$, let ~$m_j$  be the smallest integer such that~$PX_{n-1,j}=PX_{m_j,j}=j$.
 
The left standardization of $X$ is   the triangle $LX$  of size~$n-1$ obtained as follows:

\begin{eqnarray}
LX_{i,j}&=&PX_{i,j}=j\quad\text{ for } \quad  j\leq k \quad\text{ and }\quad i\geq m_j.\\
LX_{i,j}&=&PX_{i,j}-1\quad\text{ for other values of $i,j$}.
\end{eqnarray}

\smallskip

\subsubsection{Right standardization}
Let $X$ be a Gog triangle of size $n$, let $PX$ be its projection, with upper row of the form
$1,2,\ldots,k,k+2,\ldots,n$, and for $j\geq k+1$
let ~$p_j\geq 1$  be the largest integer such that~$PX_{n-1,j}=PX_{n-p_j,j+1-p_j}=j+1$.

The right standardization of $X$ is   the triangle $RX$ of size~$n-1$ obtained as follows:

\begin{eqnarray}
\label{eq::stdl2}
RX_{n-l,j+1-l}&=&PX_{n-l,j+1-l}-1=j\quad\text{ for } \quad  j\geq k+1 \quad\text{ and }\quad 1\leq l\leq p_j.\\
RX_{i,j}&=&PX_{i,j}\quad\text{ for other values of i,j}.
\end{eqnarray}

\smallskip

The right standardization of $X$ can be obtained by reflecting $X$ vertically, according to 
(\ref{reflection}), applying left standardization, and then reflecting vertically again.

Below are a Gog triangle of size 6,
  together with  its left and right standardizations:
\bigskip
\begin{center}
 \begin{tikzpicture}[scale=.4]

\draw  (0,4) node{1} ;\draw  (2,4) node{2} ;\draw  (4,4) node{3} ;\draw  (6,4) node{4} ;\draw  (8,4) node{5} ;\draw (10,4) node{6} ;
     \draw  (1,3) node{1} ;\draw  (3,3) node{2} ;\draw  (5,3) node{3} ;\draw  (7,3) node{5} ;\draw  (9,3) node{6} ;
      \draw  (2,2) node{1} ;\draw  (4,2) node{3} ;\draw  (6,2) node{5} ;\draw  (8,2) node{6} ;
       \draw  (3,1) node{1} ;\draw  (5,1) node{4} ;\draw  (7,1) node{6} ;
        \draw  (4,0) node{3} ;\draw  (6,0) node{5} ;
         \draw  (5,-1) node{4} ;
                     \draw(5,-3)node{$X$};
       
         \draw  (10.5,1) node{,} ;
  \end{tikzpicture}
\end{center}
\begin{center}
    \begin{tikzpicture}[scale=.4]
    \begin{scope}
    \draw[ClasseR,rotate=45](2.8,0)ellipse(6mm and 22mm);
    \draw[ClasseR](3,3)ellipse(6mm and 6mm);
    \draw[ClasseR](5,3)ellipse(6mm and 6mm);
    \draw[ClasseB,rotate=-45](4.3,7)ellipse(6mm and 22mm);
    \draw[ClasseB,rotate=-45](2.8,6.4)ellipse(6mm and 16mm);
     \draw  (1,3) node{1} ;\draw  (3,3) node{2} ;\draw  (5,3) node{3} ;\draw  (7,3) node{5} ;\draw  (9,3) node{6} ;
      \draw  (2,2) node{1} ;\draw  (4,2) node{3} ;\draw  (6,2) node{5} ;\draw  (8,2) node{6} ;
       \draw  (3,1) node{1} ;\draw  (5,1) node{4} ;\draw  (7,1) node{6} ;
        \draw  (4,0) node{3} ;\draw  (6,0) node{5} ;
         \draw  (5,-1) node{4} ;
                     \draw(5,-3)node{$PX$};
         \draw  (10.5,1) node{,} ;
  \end{scope}
  \begin{scope}[xshift=11cm]
  \draw[color=red]  (1,3) node{1} ;\draw[color=red]  (3,3) node{2} ;\draw[color=red]  (5,3) node{3} ;\draw  (7,3) node{4} ;\draw  (9,3) node{5} ;
             \draw[color=red]  (2,2) node{1} ;\draw  (4,2) node{2} ;\draw  (6,2) node{4} ;\draw  (8,2) node{5} ;
                   \draw[color=red]  (3,1) node{1} ;\draw  (5,1) node{3} ;\draw  (7,1) node{5} ;
                         \draw  (4,0) node{2} ;\draw  (6,0) node{4} ;
                         \draw  (5,-1) node{3} ;
                         \draw(5,-3)node{$LX$};
         \draw  (10.5,1) node{,} ;
  \end{scope}
    \begin{scope}[xshift=22cm]
    \draw  (1,3) node{1} ;\draw  (3,3) node{2} ;\draw  (5,3) node{3} ;\draw[color=blue]  (7,3) node{4} ;
\draw[color=blue]  (9,3) node{5} ;
             \draw  (2,2) node{1} ;\draw  (4,2) node{3} ;\draw[color=blue]   (6,2) node{4} ;\draw[color=blue]   (8,2) node{5} ;
                   \draw  (3,1) node{1} ;\draw   (5,1) node{4} ;\draw[color=blue]   (7,1) node{5} ;
                         \draw  (4,0) node{3} ;\draw   (6,0) node{5} ;
                         \draw  (5,-1) node{4} ;
                         \draw(5,-3)node{$RX$};
  \end{scope}
  \end{tikzpicture}
    \end{center}

\subsubsection{}
\begin{prop}
Let $X$ be a Gog triangle of size $n$, with $l$ inversions and $m$ coinversions, and such that the top row of $PX$ is $1,2,3,\ldots,k,k+2,\ldots n$.

 Then $LX$
   is a Gog triangle of size $n-1$ with    $m-n+k+1$
coinversions and at most $l$  inversions.

Similarly, $RX$ is a Gog triangle of size $n-1$ with  $l-k$  inversions   and at most  $m$
coinversions.
\end{prop}
\begin{proof}
 We prove the proposition only for the left standardization. 
 The case of right standardization can be proven in an analogous way.
 \smallskip
 
 Observe first that there are exactly $k$ inversions and $n-k-1$ coinversions on row $n-1$ of $X$. It follows that $PX$ has $l-k$ inversions and $m-n+k+1$ coinversions.

\smallskip

We first prove that $LX$ is a Gog triangle. For any $i,j$ we have to prove that 
$$LX_{i,j}\geq LX_{i-1,j-1}, LX_{i,j}\geq X_{i+1,j},LX_{i,j}> LX_{i,j-1}. $$
Since $PX$ comes from a Gog triangle one has $PX_{i,j}\geq PX_{i-1,j-1}$, therefore the first  inequality may fail only if 
$LX_{i,j}=PX_{i,j}-1$ and $LX_{i-1,j-1}=PX_{i-1,j-1}$. If this is the case then
$PX_{i,j}>j$ and 
$PX_{i-1,j-1}=j-1$, therefore $LX_{i,j}> LX_{i-1,j-1}$. This shows also   that 
 $LX$ has the same coinversions as $PX$, so their number is $m-n+k+1$. 
A similar reasoning yields the other inequalities, moreover the number of inversions of $LX$ can increase at most by $k$ with respect to that of $PX$  (more precisely by at most one in each of the $k$ leftmost NW-SE diagonals). 
\end{proof}

\subsubsection{}
We can now finish the proof of Theorem \ref{diamond} by induction on $n$. For $n=3$ or $4$, the claim follows by inspection of Table 1.
Let $X$ be a Gog triangle of size $n$, with $l$ inversions and $m$ coinversions. 
We have to prove that $(l,m)$ belongs to some $A_{n,r}$.
We have seen that 
$LX$ is a Gog triangle of size $n-1$ with $m-n+k+1$ coinversions and at most $l$ inversions, whereas $RX$ is a Gog triangle of size $n-1$ with $l-k$ inversions at most $m$ coinversions.
By induction hypothesis there exists some $p$ such that 
\begin{equation}\label{pp}l\geq \frac{p(p+1)}{2},\qquad m-n+k+1\geq \frac{(n-p-2)(n-p-1)}{2}
\end{equation}
 and there exists $q$ such that
\begin{equation}\label{qq}l-k\geq \frac{q(q+1)}{2},\qquad m\geq \frac{(n-q-2)(n-q-1)}{2}.\end{equation}

If $p>q$, then (\ref{pp}) implies $l\geq \frac{(q+1)(q+2)}{2}$
and since $m\geq \frac{(n-q-2)(n-q-1)}{2}$ by (\ref{qq}) one has 
$(l,m)\in A_{n,q+1}$. 

Similarly if $q<p$ then 
(\ref{qq}) implies $m\geq \frac{(n-p-1)(n-p)}{2}$  and $l\geq \frac{p(p+1)}{2}$
 by (\ref{pp}) so that  
 $(l,m)\in A_{n,p}$.

If now $p=q$ then either $k> p$ and then $l-k\geq \frac{q(q+1)}{2}$
 implies  $l\geq \frac{(p+1)(p+2)}{2}$
and $(l,m)\in A_{n,p+1}$,
or $k\leq p$ then 
$m-n+k+1\geq \frac{(n-p-2)(n-p-1)}{2}$
implies $m\geq \frac{(n-p-1)(n-p)}{2}$ and
$(l,m)\in A_{n,p}$.

Suppose now that
$l=\frac{p(p+1)}{2}, m=\frac{(n-p-1)(n-p)}{2}$. We wish to prove that there exists a unique Gog triangle with these numbers of inversions and coinversions. Consider $RX$, which has $l-k$
inversions. If $k>p$ then  $l-k<\frac{(p-1)p}{2}$, therefore, by Remark \ref{pp+1}, $RX$ has at least 
$\frac{(n-p)(n-p+1)}{2}$ coinversions, which contradicts the fact that $RX$ has at most $m=\frac{(n-p-1)(n-p)}{2}$ coinversions; it follows that $k\leq p$. A similar reasoning with $LX$ shows that in fact $k=p$, and $RX$ has $\frac{(p-1)p}{2}$ inversions, and at most  $m=\frac{(n-p-1)(n-p)}{2}$ coinversions,
hence $RX$, agin by Remark \ref{pp+1} it has exactly $m=\frac{(n-p-1)(n-p)}{2}$, so, by induction hypothesis, it is the unique  Gog triangle with $l=\frac{(p-1)p}{2}$
inversions, and  $\frac{(n-p-1)(n-p)}{2}$ coinversions. Comparing with the formula
(\ref{trdiamond}) for this triangle, we find that $X$ is again the unique triangle with $l=\frac{p(p+1)}{2}, m=\frac{(n-p-1)(n-p)}{2}$.
\qed

\end{document}